\documentclass[10pt,letterpaper]{article}

\usepackage{amsmath}
\usepackage{graphicx}
\usepackage{amsfonts}
\usepackage{amssymb}
\usepackage[letterpaper]{geometry}
\usepackage{comment}
\usepackage{color}
\newtheorem{theorem}{Theorem}

\newtheorem{lemma}[theorem]{Lemma}

\newenvironment{proof}[1][Proof]{\noindent\textbf{#1.} }{\ \rule{0.5em}{0.5em}}
\numberwithin {equation}{section}
\numberwithin {theorem}{section}

\begin{document}

\title{Quantizing the geodesic flow via \\adapted complex structures}
\author{William D. Kirwin\thanks{Mathematics Institute, University of Cologne,
Weyertal 86 - 90, 50931 Cologne, Germany.\newline email:
\texttt{will.kirwin@gmail.com}}}
\date{}
\maketitle

\begin{abstract}
The geometric quantization of the geodesic flow on a compact Riemannian manifold via the BKS ``dragging projection'' yields the Laplacian plus a scalar curvature term. To avoid convergence issues, the standard construction involves somewhat unnatural hypotheses that do not hold in typical examples. In this paper, we use adapted complex structures to make sense of a Wick-rotated version of the dragging projection which avoids the convergence issues.
\end{abstract}

\section{Introduction}

The dynamics of a free particle moving on a Riemannian manifold $(M,g)$ is modeled by the geodesic flow on the cotangent bundle $T^{\ast}M$, which is itself the Hamiltonian flow of the free-particle kinetic energy function $E:=\frac{1}{2}\left\vert p\right\vert ^{2}$. Quantizing $E$ in ``position'' space $\mathcal{H}:=L^2(M,d\mathrm{vol}_g)$ then yields an operator $Q(E):\mathcal{H}\rightarrow\mathcal{H}$ called the (quantum) Hamiltonian. Computing $Q(E)$ turns out, however, to be a surprisingly subtle art and there does not seem to be a unique answer, although all techniques yield an operator of the form
\begin{equation}
\label{eqn:QE}
Q(E)=-\frac{\hbar^{2}}{2}\left(\Delta-cS\right),
\end{equation}
where $\Delta$ is the $g$-Laplacian on $M$, $S$ is the scalar curvature of the metric $g$, and $c$ is an essentially author-dependent (or, perhaps more precisely, quantization-technique-dependent) constant\footnote{The particular value of the constant $c$ seems to be related to a choice of operator ordering, although we do not know the precise dependence. In \cite{Douglas-Klevtsov}, for example, one finds a value of $c=1/2$ using normal ordering. In this paper, we will use ``position-space'' geometric quantization, which should correspond to the anti-Kohn-Nirenberg ordering, for which the value of $c$ is $1/6$. It seems that an entire continuous range of values is possible and many appear in the literature \cite{Bastianelli-VanNieuwenhuizen,Douglas-Klevtsov,Driver-Andersson,Baer-Pfaeffle,DeWitt,Fulling,Prat-Waldron}.}, which in our case will be $c=1/6$.

One standard way of computing $Q(E)$ in geometric quantization is via the so-called ``dragging projection'' construction \cite{Woodhouse,SniatyckiGQ}. This construction, as it is usually presented, is riddled with convergence difficulties (which can likely be overcome with sufficient dedication); we explain them in detail in Section \ref{sec:probs}. In this paper, we will use recent techniques involving adapted complex structures to make rigorous sense of a Wick-rotated version of the dragging projection which avoids all of these difficulties.

We begin with a brief description of the dragging projection and the associated difficulties, after which we explain how adapted complex structures and the idea of complex-time evolution alleviate these difficulties, admittedly at the expense of introducing some more advanced geometric machinery.

The position-space quantization of $(M,g)$ is, according to geometric quantization, the space of sections of a certain line bundle over $T^{\ast}M$ that are covariantly constant along the fibers of $\pi:T^{\ast}M\rightarrow M$. To quantize the geodesic flow, one lifts the geodesic flow to this line bundle to obtain a ``quantized'' flow on the space of vertically constant sections, but one immediately sees that the lifted flow does not actually preserve the property ``covariantly constant along the fibers''.

If $P$ is a (complex integrable) Lagrangian distribution on $T^{\ast}M$ which is transverse to the vertical tangent bundle, there is a canonical nondegenerate pairing, called the BKS (Blattner--Kostant--Sternberg) pairing, between sections of our line bundle which are covariantly constant along $P$ and sections which are covariantly constant in the vertical directions. The BKS pairing therefore induces a projection from $P$-constant sections to vertically constant sections.

The ``dragging projection'' construction of $Q(E)$ is as follows: 1) lift the geodesic flow to the line bundle, 2) flow a vertically constant section for a short time, and 3) project back to the space of vertically constant sections using the BKS pairing.

The convergence problems arise when one tries to actually compute the BKS projection. First, there is no guarantee that the pushforward by the geodesic flow of the vertical tangent bundle yields a distribution which is transverse to the vertical tangent bundle at every point. Of course, the set of points where the two distributions are not transverse should be measure zero and one can probably just ignore them, but a little care has to be taken with how the BKS pairing behaves near such points.

Second, when one simply writes down the BKS pairing between vertically constant sections and the geodesically flowed sections, the integral which results is divergent. It  is then necessary to introduce some regularization procedure to make sense of the integral.

Finally, we want to do this for ``short'' times, which means that we do it for arbitrary (small) time $t$ and then take the derivative at $t=0$. The usual approach is to use a stationary phase approximation to compute an asymptotic expansion of the BKS pairing as $t\rightarrow0$ and then read off the coefficient of $t$. But in order to even apply a stationary phase approximation, one must show that the tails of the integral vanish. Given that the integral is divergent and must first be regularized, applying stationary phase becomes a delicate issue.

It is likely that all of these issues can be dealt with, and even our construction, which as we now explain has none of these problems, is formally very similar to the classical construction (and of course yields the same ``answer'').

The fundamental observation of this article is that a Wick rotation, that is, replacing $t$ by $it$ (more precisely, analytically continuing in $t$ and evaluating at time $it$) eliminates all of the problems in the usual dragging projection construction of $Q(E)$. The price to pay is that one cannot simply replace $t$ by $it$ in all the formulas and proceed as usual; one must first make sense of the ``imaginary-time geodesic flow''. Fortunately, this has already been done: the pushforward of the vertical tangent bundle by the ``time-$i$ geodesic flow'' yields the $(1,0)$-tangent bundle of the adapted complex structure of Guillemin--Stenzel/Lempert--Sz\H{o}ke, and the ``time-$it$ flow'' can be understood as the pushforward of the adapted complex structure by rescaling in the fibers of $T^*M$ by $t$. We explain more about the geometry of these ``imaginary-time flows'' in Section \ref{sec:acs}.

We would be remiss if we did not mention other approaches to quantizing the geodesic flow on $(M,g)$. In \cite{DeWitt}, DeWitt finds a value of $1/6$ using physics path integral techniques. More recently, also using path integral techniques from physics, Bastianelli and Van Nieuwenhuizen \cite{Bastianelli-VanNieuwenhuizen}, and then Douglas and Klevtsov in the K\"{a}hler setting \cite{Douglas-Klevtsov}, compute that the quantum Hamiltonian of a free particle in the Bargmann-Fock representation is $-\hbar^{2}\Delta_{\bar{\partial}}+\frac{\hbar^{2}}{4}S$ (that is, they obtain $c=1/2$). In fact, they explain that their value of $c$ is related the choice of symmetric ordering, which is itself associated with Toeplitz quantization rather than position-space quantization, and, in particular, $c=1/2$ corresponds to how one defines the Heavyside step-function at $0$ in the propagators of the theory. The mathematically rigorous path integral computations of Driver and Andersson \cite{Driver-Andersson}, and later B\"ar und Pf\"affle \cite{Baer-Pfaeffle}, show that the constant $c$ is related to a choice of measure on the spaces of geodesic polygons which they use to approximate the path integral and can take an entire continuous range of values. Letting $\Lambda=-1/2$ in \cite[Thm 5.2]{Baer-Pfaeffle} yields our result $c=1/6$. In \cite{Fulling}, by using various powers of the Van-Vleck--Morette determinant in the Weyl quantization of the geodesic flow, Fulling argues that the constant $c$ appearing in (\ref{eqn:QE}) should be $0,$ $1/3$ or $1/6$, or really any value between $0$ and $1/6$, and that for various reasons $c=1/6$ might be the most ``natural'' value. In more recent work \cite{Prat-Waldron}, Prat-Waldron shows that there are corrections arising in the supersymmetric version of the dragging projection construction of $Q(E)$ which conspire to yield $c=1$.

The fact that a Wick rotation eliminates various convergence issues is far from unique to our setting. It is standard in mathematical physics to work in Euclidean time, that is, after a Wick rotation, for just this reason. Indeed, the path integrals considered in \cite{Bastianelli-VanNieuwenhuizen}, \cite{Douglas-Klevtsov}, \cite{Driver-Andersson}, and \cite{Baer-Pfaeffle} are all ``Euclidean-time'' path integrals.

\bigskip
The rest of the paper is organized as follows. We begin in Section \ref{sec:hfq} with a brief summary of half-form quantization before turning in Section \ref{sec:dragproj} to a more technical description of the usual construction of the dragging projection as found in \cite{SniatyckiGQ} or \cite{Woodhouse}; in particular, in Section \ref{sec:probs}, we give a precise account of the technical difficulties which arise. In Section \ref{sec:wr} we briefly discuss Wick rotation. In Section \ref{sec:acs}, we recall the results of the author and B. Hall which allow us to interpret the ``imaginary-time geodesic flows'' in terms of adapted complex structures. We conclude in Section \ref{sec:WrDP} with the precise quantization of the geodesic flow via a Wick-rotated dragging projection. Although many of the steps are formally similar to those of the classical construction, we include them both for completeness and to demonstrate the resolution of the analytical difficulties which occur in the standard construction.

\section{Half-form quantization} \label{sec:hfq}

Consider a symplectic manifold $(N,\omega)$ of real dimension $2n$. The complex Lagrangian Grassmannian of $N$ is
\[
\mathcal{L}_{\mathbb{C}}(N):=\{(x,P_{x}):x\in N,P_{x}\subset T_{x}^{\mathbb{C}}(N)\text{ Lagrangian}\}\overset{\varpi}{\rightarrow}N.
\]
A section $P\in\Gamma(\mathcal{L}_{\mathbb{C}}(N))$ is called a polarization if for each $x\in N$, $[P_{x},P_{x}]\subset P_{x}$, that is, if $P$ is an
integrable complex Lagrangian distribution.

A Lagrangian subspace $P_{x}\in T_{x}^{\mathbb{C}}(N)$ is said to be nonnegative if for each $x\in N$, $P_{x}$ is nonnegative in the sense that for each $Z\in P_{x},$ we have $-2i\omega(Z,\bar{Z})\geq0$ (here, $\omega$ is extended complex-linearly to the complexified tangent space). Denote the nonnegative complex Lagrangian Grassmannian of $N$ by $\mathcal{L}_{\mathbb{C}}^{+}(N)$. The fiber $\varpi^{-1}(x)$ is a contractible homogeneous space which can be parameterized by the Siegel upper half-space.

A metaplectic structure on $N$ is a Hermitian line bundle $\delta\rightarrow\mathcal{L}_{\mathbb{C}}^{+}(N)$ with compatible connection equipped with an isomorphism $\delta\otimes\delta\rightarrow\bigwedge^{n}\mathcal{L}_{\mathbb{C}}^{+}(N)$. A metaplectic structure exists on $N$ if and only if the second Stiefel--Whitney class of $N$ vanishes. In this paper, we will only be interested in certain $1$-parameter families of polarizations, so we will not dwell much on general constructions involving $\delta.$

For a polarization $P$, we set $P_{x}^{\ast}:=\{\left.  \alpha\in T_{x}^{\ast}N\right\vert X\,$\raisebox{0.44ex}{\rule{4pt}{0.35pt}}$\!\overset{\lrcorner}{}\,\alpha=0$ for all $X\in\bar{P}_{x}\}.$ The canonical bundle of $P$ is $\mathcal{K}^{P}:=\bigwedge^{n}P^{\ast}.$ If $P$ is everywhere positive, it defines a complex structure on $N$ by declaring the $(1,0)$-tangent space to be $P$, and in this case the canonical bundle of $P$ is the usual canonical bundle of $N$ as a complex manifold.

The pullback of the metaplectic structure along $P,$ called the bundle of $P$-half-forms (or just half-forms, if $P$ is clear from the context) is a line bundle $\delta^{P}:=P^{\ast}\delta\rightarrow N$ which is a square root of $\mathcal{K}^{P}$. It has a canonical Hermitian metric known as the Blattner--Kostant--Sternberg (BKS) pairing \cite{Woodhouse}.

The definition of the BKS pairing is simplest when $P\cap\bar{P}=\{0\}$. The Liouville volume form on $N$ is
\begin{equation}
\label{eqn:Lvol}
\varepsilon:=(-1)^{n(n-1)/2}\frac{\omega^{n}}{(2\pi\hbar)^{n}n!}.
\end{equation}
If $\mu$ and $\mu^{\prime}$ are sections of $\delta^{P}$, then $(\mu,\mu^{\prime})$ is the unique function on $N$ determined by
\begin{equation}
\label{eqn:BKS}
\bar{\mu}^{2}\wedge(\mu^{\prime})^{2}=i^{n}(\mu,\mu^{\prime})^{2}\varepsilon.
\end{equation}
(One should worry a bit about the sign of the square root, but as we will see, there is an obvious choice for the polarizations we are concerned with and the general structure will play no role in what follows.)

The canonical bundle $\mathcal{K}^{P}$ admits a partial connection given by $\nabla=X\,$\raisebox{0.44ex}{\rule{4pt}{0.35pt}}$\!\overset{\lrcorner}{}\,d$ for $X\in\bar{P}$, and the half-form bundle $\delta^{P}$ inherits a partial connection from that on $\mathcal{K}^{P}$ via $2\mu\otimes\nabla^{\delta^{P}}\mu:=\nabla^{\mathcal{K}^{P}}(\mu^{2}).$ In particular, we say that a (possibly local) section $\mu$ of $\delta^{P}$ is covariantly constant along $\bar{P}$, or \emph{polarized,} if $\nabla_{X}^{\delta^{P}}\mu=0$ for all $X\in\bar{P}.$

The BKS pairing\ (\ref{eqn:BKS}) extends to a nondegenerate pairing between sections of $\delta^{P}$ and $\delta^{P^{\prime}}$ so long as $\bar{P}\cap P^{\prime}=\{0\}.$ Since the pairing is nondegenerate, it induces a map $\Pi_{P,P^{\prime}}:\Gamma(\delta^{P^{\prime}})\rightarrow\Gamma(\delta^{P})$ called the BKS projection. When $\bar{P}\cap P^{\prime}\neq\{0\}$, the definition is slightly more involved and we refer the reader to \cite[page 231]{Woodhouse}, although we point out that in general, the BKS pairing takes values in the set of densities on $T^{\mathbb{C}}N/(\bar{P}\cap P^{\prime}).$

Fix $\hbar>0$ such that $[\omega/2\pi\hbar]$ is integral. Let $L\rightarrow N$ be a Hermitian line bundle with connection $\nabla$ with curvature $-i\omega/\hbar$; that is, the bundle $L$ is a prequantum bundle for $(N,\omega).$ Given a polarization $P$, the Hermitian structure on $L$ and the BKS\ pairing on $\delta^{P}$ combine to yield a pointwise Hermitian structure on $L\otimes\delta^{P}$ determined by $\left(  \psi^{\prime}\otimes\mu^{\prime},\psi\otimes\mu\right)(x) =(\psi^{\prime},\psi)_{L}(x)(\mu^{\prime},\mu)_{\text{BKS}}(x).$ This pointwise Hermitian structure then yields a Hermitian inner product on $\Gamma(L\otimes\delta^{P})$ by integrating against the Liouville form $\varepsilon.$

Given a polarization $P$, the (half-form corrected) quantum Hilbert space $\mathcal{H}_{P}$ associated to $P$ is defined to be the $L^{2}$-closure of the space of smooth sections of $L\otimes P^{\ast}\delta$ which are covariantly constant along $P$, where the inner product is induced from the Hermitian structure on $L$ and the BKS pairing:
\[
\mathcal{H}_{P}:=L_{P}^{2}(N,L\otimes P^{\ast}\delta).
\]

The Kostant--Souriau\footnote{The Kostant--Souriau quantization $f\mapsto\hat{f}$ is a linear map which satisfies $[\hat{f},\hat{g}]=i\hbar\widehat{\left\{  f,g\right\}  }$ and $\hat{1}=1$ and hence comes close to satisfying Dirac's definition of quantization. However, the prequantum Hilbert space---the completion of $\Gamma(L)$ with respect to the inner product induced by the Hermitian structure on $L$---is too large, and one would like to restrict $\hat{f}$ to an operator on $\mathcal{H}_{P}$. But in general, the vector field $X_{f}$ does not preserve $P$ and hence, even if one extends the partial connection $\nabla^{\delta^{P}}$ to a full connection (since it is not defined if $X_{f}$ does not preserve $P$), the operator $\hat{f}$ defined by (\ref{eqn:KSq}) does not map polarized sections to polarized sections. Of course, if $f$ is tangent to $P$ (that is, $df$ vanishes on $P$) and $X_{f}$ preserves $P$ (that is, $[X_{f},P]\subset P$), then $\hat{f}$ is a well-defined operator on $\mathcal{H}_{P}$, but this turns out not to be the case for most interesting observables; in particular, the geodesic flow does not preserve any polarization.} quantization of a classical observable $f\in C^{\infty}(N)$ is the differential operator
\begin{equation}
\label{eqn:KSq}
\hat{f}:=i\hbar\nabla_{X_{f}}+f:\Gamma(L)\rightarrow\Gamma(L),
\end{equation}
where $X_{f}$ is the Hamiltonian vector field of $f$ determined by $X_{f}\,$\raisebox{0.44ex}{\rule{4pt}{0.35pt}}$\!\overset{\lrcorner}{}\,\omega=df$. Let $iz\partial_{z}$ denote the canonical $\mathbb{C}^{\times}$-invariant Liouville vector field on the fibers of $L.$ Then
\begin{equation}
\label{eqn:Vf}
V_{f}:=X_{f}^{hor}+f\cdot iz\partial_{z}
\end{equation}
is the infinitesimal lift of the flow of $X_{f}$ to the total space of $L$. The time-$\sigma$ flow of $V_{f}$ induces a flow $\hat{\rho}_{\sigma}^{f}:\Gamma(L)\rightarrow\Gamma(L)$ on sections of $L$ which is related to the operator $\hat{f}$ by
\begin{equation}
\label{eqn:KSasFlow}
\left.  \frac{d}{d\sigma}\right\vert _{\sigma=0}\hat{\rho}_{\sigma}^{f}(s)=-\frac{i}{\hbar}\hat{f}s,\quad s\in\Gamma(L)
\end{equation}
\cite[Prop. 3.1]{CharlesPI}.

\bigskip
We now specialize to the case $N=T^{\ast}M$, where $(M,g)$ is an $n$-dimensional real-analytic Riemannian manifold with real-analytic metric $g$. Let $\theta=p_{j}dx^{j}$ denote the canonical $1$-form on $T^{\ast}M$. Since $\omega^{T^{\ast}M}=-d\theta$ is exact, we can choose $L$ to be the trivial line bundle over $T^{\ast}M$ with connection
\begin{equation}
\nabla^{L}=d+i\theta/\hbar. \label{eqn:triv1}
\end{equation}

The complexified vertical tangent bundle $P^{0}:=\mathrm{Vert}^{\mathbb{C}}(T^{\ast}M)\subset T^{\mathbb{C}}(T^{\ast}M)$ is called the vertical polarization. Denote both the volume form on $M$ and its pullback to $T^{\ast}M$ by $d\mathrm{vol}_{g}$. The canonical bundle of $P^{0}$ is trivialized by $d\mathrm{vol}_{g}$. Therefore, we can take $\delta_{P^{0}}$ to be a trivial bundle\footnote{One should worry that we make consistent choices of $\delta^{P}$ that fit together into a metaplectic structure $\delta$, but as all of the half-form bundles which appear here are pullbacks of $\delta^{P^{0}}$, we only need to insure that the metaplectic structure is chosen such that $\delta^{P^{0}}$ is trivializable.} with trivializing section $\sqrt{d\mathrm{vol}_{g}}$; that is, the isomorphism $\delta_{P^{0}} \otimes\delta_{P^{0}}\simeq\mathcal{K}^{P^{0}}$ is defined by $\sqrt {d\mathrm{vol}_{g}}\otimes\sqrt{d\mathrm{vol}_{g}}\mapsto d\mathrm{vol}_{g}$.

As $\bar{P}^{0}\cap P^{0}=P^{0}\neq\{0\}$, the formula (\ref{eqn:BKS}) is totally degenerate on $\delta^{P^{0}}$. Nevertheless, the BKS pairing is defined in this case and yields $\left\vert \sqrt{d\mathrm{vol}_{g}}\right\vert _{\text{BKS}}^{2}=d\mathrm{vol}_{g}$. By (\ref{eqn:triv1}), a trivialized section of $L$ which is covariantly constant along $P^{0}$ is just a function on $T^{\ast}M$ of the form $\psi\circ\pi$ for $\psi\in C^{\infty}(M)$, and since $X\,$\raisebox{0.44ex}{\rule{4pt}{0.35pt}}$\!\overset{\lrcorner}{}\,d(\pi^{\ast}d\mathrm{vol}_{g})=0$ for any vertical vector $X$, a polarized section of $L\otimes\delta^{P^{0}}$ is of the form $\psi\circ\pi\sqrt{d\mathrm{vol}_{g}}$. The BKS pairing therefore induces the Hermitian inner product on $\mathcal{H}_{P^{0}}$ given by
\[
\left\langle \psi^{\prime},\psi\right\rangle _{\mathcal{H}_{P^{0}}} =\frac{1}{(2\pi\hbar)^{n/2}}\int_{M}\overline{\psi^{\prime}(q)}\psi(q)d\mathrm{vol}_{g}(q)
\]
whence $\mathcal{H}_{P^{0}}\simeq L^{2}(M,d\mathrm{vol}_{g})$.

We will find it repeatedly useful to work in normal coordinates on $M.$ Given $q\in M$, let $\{x^{j}\}$ be normal coordinates in a neighborhood $U$ centered at $q$. Then the $1$-forms $dx^{j}$ trivialize $T^{\ast}U$ via $p_{j}dx^{j}\in T_{x}^{\ast}U\mapsto(x,\vec{p})\in U\times\mathbb{R}^{n}$. The Riemannian volume in local coordinates is $d\mathrm{vol}_{g}(x)=\sqrt{\det(g_{jk}(x))}d^{n}x$, and since the canonical symplectic form in these coordinates is $\omega=dx^{j}\wedge dp_{j}$, the Liouville volume form (\ref{eqn:Lvol}) is
\begin{equation}
\label{eqn:Ldecomp}
\varepsilon=\frac{1}{(2\pi\hbar)^{n}\sqrt{\det(g_{jk}(x))}}d\mathrm{vol}_{g}\wedge d^{n}p.
\end{equation}

\section{The ``dragging projection''}\label{sec:dragproj}

\subsection{The standard construction}

Denote the time-$\sigma$ geodesic flow on $T^{\ast}M\simeq_{g}TM$ by $\Phi_{\sigma}:T^{\ast}M\rightarrow T^{\ast}M$. It is the Hamiltonian flow of the kinetic energy function $E(q,p):=\frac{1}{2}g_{q}(p,p)=\frac{1}{2}\left\vert p\right\vert ^{2}$. The geodesic flow acts on polarizations by pushforward, and for any polarization $P$, $(\Phi_{\sigma})_{\ast}P\neq P$ unless $\sigma=0$. In particular, $X_{E}$ does not preserve the vertical polarization, and so the Kostant--Souriau quantization of $E$ does not define an operator on the vertically polarized Hilbert space $\mathcal{H}_{P^{0}}$.

The standard ``dragging projection'' quantization of the geodesic flow is essentially the projection of the Kostant--Souriau quantization to $\mathcal{H}_{P^{0}}$, where the projection is achieved via the BKS map as follows (see, for example, \cite[Sec. 9.7]{Woodhouse} or \cite[Sec. 6.3,7.1]{SniatyckiGQ}). The idea is to lift the geodesic flow $\Phi_{\sigma}$ to $L$ and hence to a map $\hat{\rho}_{\sigma}:L^{2}(L\otimes\delta^{P^{0}})\rightarrow L^{2}(L\otimes\delta^{P^{\sigma}}).$ The BKS pairing then induces a map $\Pi_{P,P^{\prime}}:L^{2}(L\otimes\delta^{P^{\sigma}})\rightarrow L^{2}(L\otimes\delta^{P^{0}})$ with which the time-$\sigma$ flow of a section $\psi\sqrt{d\mathrm{vol}_{g}}\in\mathcal{H}_{P^{0}}$ can be projected back into $\mathcal{H}_{P^{0}}.$ The quantization $Q(E):\mathcal{H}_{P^{0}}\rightarrow\mathcal{H}_{P^{0}}$ of $E$ is, following equation (\ref{eqn:KSasFlow}),
\[
Q(E):=i\hbar\left.  \frac{d}{d\sigma}\right\vert _{\sigma=0}\Pi_{P^{0},P^{\sigma}}\circ\hat{\rho}_{\sigma}.
\]
Since $\Pi_{P^{0},P^{\sigma}}$ is defined in terms of the BKS pairing, one actually computes $Q(E)\psi$ by evaluating, for every $\psi^{\prime}\in L^{2}(M,d\mathrm{vol}_{g})$,
\begin{equation}
 \label{eqn:QE1}
\left\langle \psi^{\prime},Q(E)\psi\right\rangle _{L^{2}(M,d\mathrm{vol}_{g})}: =i\hbar\left.  \frac{d}{d\sigma}\right\vert _{\sigma=0}\left\langle\psi^{\prime}\sqrt{d\mathrm{vol}_{g}},\hat{\rho}_{\sigma}\left(  \psi\sqrt{d\mathrm{vol}_{g}}\right)  \right\rangle _{\text{BKS}}.
\end{equation}

The first step, then, is to compute $\hat{\rho}_{\sigma}\left(  \psi\sqrt{d\mathrm{vol}_{g}}\right)  .$ Since $\theta(X_{E})=2E$, we have $-\frac{i}{\hbar}\hat{E}=X_{E}+\frac{i}{\hbar}E$ whence the lift of the geodesic flow to $L$ acts on trivialized polarized sections of $L$ by $\hat{\rho}_{\sigma}(\psi)=e^{i\sigma E/\hbar}\psi\circ\pi\circ\Phi_{\sigma}$. Since $\mathcal{K}^{P^{0}}$ is trivial, $\Phi_{\sigma}^{\ast}\mathcal{K}^{P^{0}}=\mathcal{K}^{P^{\sigma}}$ is also trivial, and we let $\delta^{P^{\sigma}}$ be the trivial line bundle with trivializing section $\sqrt{\Phi_{\sigma}^{\ast}d\mathrm{vol}_{g}}$. The geodesic flow lifts to half-forms as the map $\hat{\rho}_{\sigma}\left(  \sqrt{d\mathrm{vol}_{g}}\right) =\sqrt{\Phi_{\sigma}^{\ast}d\mathrm{vol}_{g}}$. The right-hand side of (\ref{eqn:QE1}) is therefore equal to $i\hbar\left.  \frac{d}{d\sigma}\right\vert_{\sigma}I(\sigma)$,  where
\begin{equation}
\label{eqn:Isigma}
I(\sigma):=\int_{T^{\ast}M}\overline{\psi^{\prime}\circ\pi}e^{i\sigma E/\hbar}\psi\circ\pi\circ\Phi_{\sigma}\left(  \sqrt{d\mathrm{vol}_{g}},\sqrt{\Phi_{\sigma}^{\ast}d\mathrm{vol}_{g}}\right)  _{\text{BKS}}\frac{\omega^{n}}{(2\pi\hbar)^{n}n!}.
\end{equation}

The next step is to make a change of variables in $I(\sigma)$ which puts $I(\sigma)$ into a form more useful for a stationary phase analysis.

\begin{lemma}
\label{lemma:WoodhouseRescale}The integral $I(\sigma)$ may be rewritten as
\begin{equation}
\label{eqn:IsigmaRewrite}
\sigma^{-n/2}\int_{T^{\ast}M}\overline{\psi^{\prime}\circ\pi}e^{iE/\sigma\hbar}\psi\circ\pi\circ\Phi_{1}\left(  \sqrt{d\mathrm{vol}_{g}},\sqrt{\Phi_{1}^{\ast}d\mathrm{vol}_{g}}\right)  _{\text{BKS}}\frac{\omega^{n}}{(2\pi\hbar)^{n}n!}.
\end{equation}
\end{lemma}

\begin{proof}
Let $N_{t}:T^{\ast}M\rightarrow T^{\ast}M$ be rescaling in the fibers by $t$. Then reparameterization of geodesics, considered as curves in $T^{\ast}M$, reads $N_{t}\circ\Phi_{t\sigma}=\Phi_{\sigma}\circ N_{t}.$ One has that $\pi\circ N_{t}=\pi$, $E\circ N_{t}=t^{2}E$, and that $N_{t}^{\ast}\theta=t\theta$ whence $N_{t}^{\ast}\omega^{n}=t^{n}\omega^{n}$. Applying this to the definition (\ref{eqn:BKS}) shows that
\[
\left(  \sqrt{d\mathrm{vol}_{g}},\sqrt{\Phi_{t\sigma}^{\ast}d\mathrm{vol}_{g}}\right)_{\text{BKS}}=t^{n/2}N_{t}^{\ast}\left(  \sqrt{d\mathrm{vol}_{g}}, \sqrt{\Phi_{\sigma}^{\ast}d\mathrm{vol}_{g}}\right)_{\text{BKS}},
\]
from which the lemma follows.
\end{proof}

\bigskip
Since we are interested in small $\sigma$, the integral $I(\sigma)$ as written above seems to lend itself to a stationary phase approximation. By (\ref{eqn:Ldecomp}), the integral $I(\sigma)$ becomes
\begin{equation}
\label{eqn:IsigmaDecomp}
\int_{M}\overline{\psi^{\prime}(q)}~\mathfrak{i}(q)~d\mathrm{vol}_{g},
\end{equation}
where, after noting that $\Phi_{1}$ is just the exponential map and using
(\ref{eqn:BKS}),
\begin{align*}
\mathfrak{i}(q)~  &  :=(2\pi\hbar)^{n/2}\int_{T_{q}^{\ast}M}e^{ig^{jk}(q)p_{j}p_{k}/2\sigma\hbar~}\psi(\pi\circ\exp_{x}(p))\\
&  \qquad\qquad\qquad\qquad\times\left(  \sqrt{\frac{\det(g_{jk}(\exp_{x}(p))}{\det(g_{jk}(x))}} \frac{d^{n}x\wedge d^{n}(\exp_{x}(p))}{i^{n}d^{n}x\wedge d^{n}p}\right)^{1/2}d^{n}p.
\end{align*}

For completeness, we will include the details of a very similar computation in Section \ref{sec:WrDP}, and so we simply note here that a stationary phase approximation of the integral $\mathfrak{i}(q)$ near the zero section yields an asymptotic expansion $\mathfrak{i}(q)\sim a_{0}+a_{1}\sigma+O(\sigma^{2}),\ \sigma\rightarrow0,$ with
\begin{align*}
a_{0}(q)  &  =\frac{1}{i\hbar}\psi(q)\text{, and}\\
a_{1}(q)  &  =\frac{i\hbar}{2}\left(  \Delta-\frac{S(q)}{6}\right)  \psi(q),
\end{align*}
where $\Delta$ is the metric Laplacian and $S$ is the scalar curvature of $g$. By (\ref{eqn:QE1}), we conclude that $Q(E)=-\frac{\hbar^{2}}{2}\left(\Delta-\frac{S}{6}\right).$ We refer the reader to \cite[Sec. 9.7]{Woodhouse} or \cite[Sec. 6.3,7.1]{SniatyckiGQ} for details of this computation.

\subsection{Convergence difficulties\label{sec:probs}}

There are several analytical difficulties with this construction of the dragging projection which require either the addition of hypotheses or a more involved analysis of the convergence properties of various integrals.

First of all, in our treatment of the BKS pairing, we have implicitly assumed that the pairing $\left(\sqrt{d\mathrm{vol}_{g}},\sqrt{\Phi_{1}^{\ast}d\mathrm{vol}_{g}}\right)$ is a well-defined smooth function, which is only the case if $\left((\Phi_{\sigma})_{\ast}P^{0}\right)$ is everywhere transverse to $P^{0}$. This is certainly not the case in general; indeed, if $v\in\mathrm{Vert}_{\Phi_{-t}(z)}(T^{\ast}M)$, then $\pi_{\ast}(\Phi_{t})_{\ast}v$ is a Jacobi field along the geodesic $\gamma_{z}$ in $M$ determined by $z$. This Jacobi field is zero at points where the image of $v$ under the pushforward by the geodesic flow is vertical. In particular, $(\Phi_{\sigma})_{\ast}P_{\Phi_{-\sigma}(z)}^{0}\cap P_{z}^{0}$ is nonzero if $\pi(\Phi_{-\sigma}(z))$ and $\pi(z)$ are conjugate points on $\gamma_{z}.$

On the other hand, one could argue that since the set of points where $(\Phi_{\sigma})_{\ast}P_{\Phi_{-\sigma}(x)}^{0}\cap P_{x}^{0}\neq\{0\}$ is certainly measure zero, it can be ignored when computing the integrals in (\ref{eqn:QE1}). One would then have to check that the half-form pairing remains bounded near such points, and maybe worry very slightly about taking the derivative with respect to $\sigma$ of a family of integrals whose domains depend on $\sigma$.

A more serious problem with the standard dragging projection is the convergence of $I(\sigma)$: since $\psi$ is a function on the compact manifold $M$, $\psi\circ\pi\circ\Phi_{1}(x,p)$ does not go to zero for large $p$, and so there is no reason that the integral $\mathfrak{i}(q)$ over the fiber $T_{q}^{\ast}M$ should converge. One might hope that the BKS contribution controls the divergence, but even in the simplest case $M=S^{1}$, the integral $I(\sigma)$ is
\[
\frac{1}{\sqrt{2i\sigma}}\int_{S^{1}}\bar{\psi}^{\prime}(q)\left(\int_{\mathbb{R}}\psi(q+\sigma p)e^{ip^{2}/2\sigma}dp\right)  dq,
\]
and the integral over $\mathbb{R}$ simply does not converge (as a Lebesgue integral, at least).

The standard remedy is to introduce a cutoff $R>0$, compute the integral over
\[
T^{\ast,R}M:=\{(q,p):\left\vert p\right\vert <R\},
\]
and then let $R\rightarrow\infty.$ But since we are actually interested in the derivative at $\sigma=0$ of $I(\sigma)$, one must be careful about interchanging the $R\rightarrow\infty$ and $\sigma\rightarrow0$ limits.

The final, and perhaps most serious, difficulty lies with the application of the method of stationary phase. In order to rigorously apply the stationary phase approximation, we must first introduce a cutoff $R>0$ and show that part of $I(\sigma)$ given by the integral over $T^{\ast,>R}M:=\{(q,p):\left\vert p\right\vert >R\}$ is $O(\sigma^{1+\varepsilon})$ for some $\varepsilon>0$ (so that the derivative at $0$ is equal to $0$). But as we explained above, this part of $I(\sigma)$ does not even converge, so we must introduce another cutoff $R^{\prime}>R>0$ to regularize, then show that the $R^{\prime}\rightarrow\infty$ limit yields something of order $O(\sigma^{1+\varepsilon})$ as $\sigma\rightarrow0$.

It is likely that with sufficient patience and analytic skill, all of these subtleties can be dealt with, but we will take a different approach which simultaneously avoids all of these problems.

\section{Wick rotation}
\label{sec:wr}

Our basic idea is a common one in physics: we will do a so-called Wick rotation. Consider, for example, the standard Gaussian integral
\begin{equation}
\label{eqn:Gaussian}
\int_{\mathbb{R}^{2}}e^{-\tau(x^{2}+y^{2})}dx\,dy.
\end{equation}
For $\tau$ real and positive, one may evaluate the integral exactly to obtain $\pi/\tau$. Away from $\tau=0$, this is analytic in $\tau$, and we can analytically continue. In particular, for $\tau=-i\sigma,$ we obtain the formal expression
\begin{equation}
\int_{\mathbb{R}^{2}}e^{i\sigma(x^{2}+y^{2})}dx\,dy=i\pi/\sigma.
\label{eqn:Wick1}
\end{equation}
Of course, this integral does not converge (at least, as a Lebesgue integral), but introducing a cutoff, computing, and taking the limit as the cutoff goes to infinity yields exactly the value $i\pi/\sigma$ (alternatively, we can just interpret the integral as a Riemannian integral). On the other hand, if we are unwilling or unable to evaluate (\ref{eqn:Wick1}), we apply the procedure in reverse, that is, replace
\begin{equation}
\label{eqn:WickRotation}
\sigma\mapsto it
\end{equation}
in (\ref{eqn:Wick1}) to obtain the Gaussian integral (\ref{eqn:Gaussian}). The replacement (\ref{eqn:WickRotation}), which amounts to rotating the time variable through an angle $\pi/2$ in the complex plane, is called a \emph{Wick rotation}.

Doing a Wick rotation $\sigma\mapsto it$ to the integral $I(\sigma)$ given by (\ref{eqn:Isigma}) and applying the same rescaling argument of Lemma \ref{lemma:WoodhouseRescale} yields the integral
\begin{equation}
\label{eqn:WickJ}
J(t)=t^{-n/2}\int_{T^{\ast}M}\overline{\psi^{\prime}\circ\pi}~\psi\circ\pi\circ\Phi_{i}~e^{-E/t\hbar}\left(  \sqrt{d\mathrm{vol}_{g}},\sqrt{\Phi_{i}^{\ast}d\mathrm{vol}_{g}}\right)  _{\text{BKS}}\frac{\omega^{n}}{(2\pi\hbar)^{n}n!}.
\end{equation}
This integral is superficially much easier to deal with since the exponential is now a Gaussian in the fibers of $T^{\ast}M$. The price we pay, however, is that we must somehow interpret the ``imaginary-time'' geodesic flow. In particular, we must make sense of functions of the form
\[
\psi\circ\pi\circ\Phi_{i}%
\]
and of the $n$-form $\Phi_{i}^{\ast}d\mathrm{vol}_{g}$. Fortunately, this has already been done and involves the notion of adapted complex structures.

\section{Adapted complex structures\label{sec:acs}}

Recall that we denote the time-$\sigma$ geodesic flow by $\Phi_{\sigma}:T^{\ast}M\rightarrow T^{\ast}M$, and that this is the flow of the Hamiltonian vector field $X_{E}$ of the kinetic energy function $E(q,p):=\frac{1}{2}g_{q}(p,p).$

There are several ways to understand the analytic continuation of the geodesic flow. The most direct way to analytically continue the geodesic flow on $M$ is to regard $M$ as sitting inside a Bruhat--Whitney/Grauert complexification $X$ and then literally analytically continue $\pi\circ\Phi$ to a map from $D\times M$ into $X$ for some disk $D\subset\mathbb{C}$.

In \cite{Lempert-Szoke,Szoke} and \cite{Guillemin-Stenzel1,Guillemin-Stenzel2}, Lempert and Sz\H{o}ke and, independently and essentially simultaneously, Guillemin and Stenzel realized that with the help of the metric $g$, a neighborhood of $M$ in $X$ can be identified with some tubular neighborhood $T^{\ast,R}M$ of $M$ in its cotangent bundle. The pushforward of the complex structure on $X$ by this identification yields a complex structure on $T^{\ast,R}M,$ called the \emph{adapted complex structure}, which is K\"{a}hler with respect to the $\omega^{T^{\ast}M}$.

Theorem \ref{thm:ACS} below, due to the author and B. Hall, repackages the construction of the adapted complex structure explicitly in terms of the geodesic flow and avoids completely the introduction of the abstract complexification $X$ (although our description below in terms of the geodesic flow is contained implicitly, in the original works of Lempert--Sz\H{o}ke and Guillemin--Stenzel), and yields exactly what we need to make sense of the various terms in (\ref{eqn:WickJ}) arising from the Wick rotation.

Since we are interested in how the geodesic flow acts on vertically polarized sections, and on the vertical polarizations itself, we begin with a formulation in terms of the vertical polarization.

\begin{theorem}
\label{thm:ACS}
For every $\varepsilon>0$ there exists $R>0$ such that for all $x\in T^{\ast,R}M$, the map
\[
\sigma\in\mathbb{R\mapsto}\left(  \Phi_{\sigma}\right)  _{\ast}\mathrm{Vert}_{x}^{\mathbb{C}}T^{\ast}M\subset T_{x}^{\mathbb{C}}(T^{\ast}M)
\]
can be analytically continued to the disk $D_{1+\varepsilon}:=\{\sigma +i\tau\in\mathbb{C}:\left\vert \sigma+it\right\vert <\varepsilon\}.$ Moreover, for all $t\in(0,1+\varepsilon),$ the distribution $(\Phi_{it})_{\ast}\mathrm{Vert}^{\mathbb{C}}(T^{\ast,R}M)$ is the $(1,0)$-tangent bundle for a complex structure $J_{t}$ such that $(T^{\ast,R}M,\omega^{T^{\ast}M},J_{t})$ is K\"{a}hler with global K\"{a}hler potential $\kappa=t\left\vert p\right\vert^{2}.$
\end{theorem}

The above theorem tells us in particular that for some $R>0$, it makes sense to ``plug in $t=\sqrt{-1}$''. The resulting complex structure is the adapted complex structure of Guillemin--Stenzel/Lempert--Sz\H{o}ke. The ``time-$it$'' adapted complex structures can be obtained from the adapted complex structure by rescaling the fibers of $T^{\ast}M$ by $1/t$ (this boils down to reparamerization of geodesics). In our application, we will be interested in small $t$ and hence large $R$. Theorem \ref{thm:ACS} can be restated in terms of $J_{t}$-holomorphic functions as follows.

\begin{theorem}
For $r\in(0,R),$ a function $f_{\mathbb{C}}:T^{\ast,r}M\rightarrow\mathbb{C}$ is $J_{t}$-holomorphic if and only if there exists a real-analytic function $f:M\rightarrow\mathbb{C}$ such that for all $x\in T^{\ast,r}M$
\[
f_{\mathbb{C}}(x)=f\circ\pi\circ\Phi_{it}(x),
\]
where the right-hand side of the above equation is understood to be the evaluation at $\sigma=it$ of the analytic continuation of the map $\sigma \in\mathbb{R}\mapsto f\circ\pi\circ\Phi_{\sigma}(x)$ to $D_{\varepsilon}.$
\end{theorem}

Since the zero section consists of fixed points of the geodesic flow, the function $f$ in this theorem is simply $f_{\mathbb{C}}$ evaluated on the zero section. This theorem can also be stated at the level of holomorphic sections of our prequantum bundle $L\rightarrow T^{\ast}M$. The lift $V_{E}$ of the vector field $X_{E}$ to the total space of $L$ is $V_{E}=X_{E}^{hor}+(E-\theta(X_{E}))iz\partial_{z}$ (c.f. (\ref{eqn:Vf})). The flow $\Phi^{L}$ of $V_{E}$ is the lift of $\Phi$ to $L$ and induces an action on sections given by
\[
(\rho_{\sigma}s)(x):=\Phi_{-\sigma}^{L}(s(\Phi_{\sigma}x)),\quad s\in\Gamma(L).
\]

Since $L$ is a Hermitian line bundle with connection, the complex structure $J_{t}$ induces a holomorphic structure on $L$ which, for example, can be described by declaring that a section $s\in\Gamma_{T^{\ast,R}M}(L)$ is holomorphic if
\[
\nabla_{(1+iJ_{t})X}s=0
\]
for all vector fields $X.$

The next theorem is a corollary to Theorem \ref{thm:ACS}.

\begin{theorem}
With respect to the trivialization $\nabla=d+\frac{i}{\hbar}\theta$, a section $s\in\Gamma_{T^{\ast,r}M}(L)$ (here, $r\leq R$) is $J_{t}$-holomorphic if and only if there exists a real-analytic function $\psi:M\rightarrow\mathbb{C}$, and hence a $J_{t}$-holomorphic function $\psi\circ\pi\circ\Phi_{it}$, such that for all $x\in T^{\ast,r}M$, \[
s(x)=\psi\circ\pi\circ\Phi_{it}(x)~e^{-tE(x)/\hbar}.
\]
\end{theorem}

\section{Quantization via Wick-rotated ``dragging projection''}
\label{sec:WrDP}

We are now ready to quantize the geodesic flow via a Wick-rotated dragging projection. Since we are unaware of a detailed analysis of even the standard dragging projection in the literature, we will include the details of the computation, despite the fact that several of the steps mirror closely what appears in \cite{Woodhouse} and \cite{SniatyckiGQ}. In particular, the rescaling argument of Lemma \ref{lemma:WoodhouseRescale} essentially appears in \cite[page 202 -- 203]{Woodhouse}, and the computations involved in the Laplace's approximation which we make below are formally similar to the stationary phase computations which appear in \cite[Sec. 9.7]{Woodhouse} and \cite[Sec. 6.3,7.1]{SniatyckiGQ}.

\bigskip
Fix $\varepsilon>0$ and let $R$ be as in Theorem \ref{thm:ACS} so that for any $t\in(0,1+\varepsilon)$, the time-$it$ adapted complex structure $J_{t}$ exists on $T^{\ast,R}M$. Strictly speaking, we should assume that $(M,g)$ is such that we can take $R=\infty$ since the integral in (\ref{eqn:WickJ}) is over $T^{\ast}M$. However, Theorem \ref{thm:tails} below shows that for small $t$, the integrand concentrates exponentially onto the zero section for small $t$, and so there is no harm in making the weaker assumption that $J_{t}$ exists only on a finite-radius tubular neighborhood of the zero section.

Since we Wick rotate the arguments given in Section \ref{sec:dragproj}, we should be careful that the usual time is now replace by $it$ in definition (\ref{eqn:QE1}). That is, suppose that $\psi\in C^{\omega}(M)$ admits a $J_{t}$-analytic continuation to $T^{\ast,r}M$ for some $r\in(0,R)$. Then via a Wick-rotated dragging projection, the quantization of the geodesic flow acting on $\psi\in L^{2}(M,d\mathrm{vol}_{g})$ is defined by
\begin{multline}
\label{eqn:QE2}
\left\langle \psi^{\prime},\tilde{Q}(E)\psi\right\rangle _{L^{2}(M,d\mathrm{vol}_{g})}  :=\hbar\left.  \frac{d}{dt}\right\vert _{t=0} \int_{T^{\ast,r}M}\overline{\psi^{\prime}\circ\pi}e^{-tE/\hbar}\psi\circ\pi\circ\Phi_{it} \\
\times\left( \sqrt{d\mathrm{vol}_{g}},\sqrt{\Phi_{it}^{\ast}d\mathrm{vol}_{g}}\right)  _{\text{BKS}}\frac{\omega^{n}}{(2\pi\hbar)^{n}n!}
\end{multline}
for all $\psi^{\prime}\in L^{2}(M,d\mathrm{vol}_{g})$, where we now understand $\psi\circ\pi\circ\Phi_{it}$ as the $J_{t}$-holomorphic function on $T^{\ast,r}M$ which is equal to $\psi$ on the zero section and recognize the combination $e^{-tE/\hbar}\psi\circ\pi\circ\Phi_{it}$ as a trivialized holomorphic section of $L\rightarrow T^{\ast,r}M$.

In order to make sense of the half-form $\sqrt{\Phi_{it}^{\ast}d\mathrm{vol}_{g}}$, note the family of polarizations $P^{it}=\left(  \Phi_{it}\right)_{\ast}P^{0},$ $t\in(0,1+\varepsilon)$, induces a family of canonical bundles $\mathcal{K}^{P^{it}},$ $t\in(0,1+\varepsilon)$, which is the analytic continuation of the family $\sigma\mapsto\Phi_{\sigma}^{\ast}\mathcal{K}^{P^{0}}$ considered as subbundles of the bundle $\bigwedge^{n}T^{\ast}(T^{\ast}M)^{\mathbb{C}}$ of $n$-forms on $T^{\ast}M.$ The bundle $\mathcal{K}^{P^{it}}$ is trivialized by the section $\Phi_{it}^{\ast}d\mathrm{vol}_{g}$, which is the analytic continuation of the family $\sigma\mapsto\Phi_{\sigma}^{\ast}d\mathrm{vol}_{g}$, and hence we can take the half-form bundle $\delta^{P^{it}}$ to be trivial with a trivializing section which we denote by $\sqrt{\Phi_{it}^{\ast}d\mathrm{vol}_{g}}$ to indicate that the isomorphism is determined by $\left(  \sqrt{\Phi_{it}^{\ast}d\mathrm{vol}_{g}}\right)^{2}=\Phi_{it}^{\ast}d\mathrm{vol}_{g}.$

By Theorem \ref{thm:ACS}, for $t>0,$ the time-$it$ polarization $P^{it}$ is everywhere transverse to the vertical polarization, and so we avoid the first of the analytic difficulties which arise in the standard dragging projection: the BKS pairing $\left(  \sqrt{d\mathrm{vol}_{g}},\sqrt{\Phi_{it}^{\ast}d\mathrm{vol}_{g}}\right)_{\text{BKS}}$ is a well-defined function everywhere on $T^{\ast,R}M$.

Since the analytic continuation of the geodesic flow has the same scaling behavior as the geodesic flow itself, the rescaling argument leading to (\ref{eqn:IsigmaRewrite}) is also valid for the Wick-rotated construction, whence the integral $J(t)$ of equation (\ref{eqn:WickJ}), which is the right-hand side of (\ref{eqn:QE2}), is equal to
\begin{equation}
\label{eqn:Jrescaled}
J(t)=t^{-n/2}\int_{T^{\ast,r}M}\overline{\psi^{\prime}\circ\pi}e^{-E/t\hbar}\psi\circ\pi\circ\Phi_{i}\left(  \sqrt{d\mathrm{vol}_{g}},\sqrt{\Phi_{i}^{\ast}d\mathrm{vol}_{g}}\right)  _{\text{BKS}}\frac{\omega^{n}}{(2\pi\hbar)^{n}n!}.
\end{equation}

We would like to apply Laplace's approximation to the integral $J(t)$. In order to do so, and to justify our claim that it is irrelevant whether we assume the adapted complex structure exists everywhere on $T^{\ast}M$ or only to a finite radius, we show in Theorem \ref{thm:tails} below that the contribution to both the integral $J(t)$ and its derivative of any region outside a neighborhood of the zero section is exponentially small. This is the essential step that is missing from the analysis of the usual dragging projection.

As usual, let $x^{j}$ be normal coordinates in a neighborhood $U$ of $q$ and let $p_{j}$ be the associated fiber coordinates of $T^{\ast}U$. At the point $q$, the metric is the identity, so if we define
\begin{equation}
\label{eqn:jtq}
j_{t}^{r}(q):=(2\pi\hbar)^{-n}t^{-n/2}\int_{B(r)}\psi\circ\pi\circ\Phi_{i}(q,p)~\left(  \sqrt{d\mathrm{vol}_{g}},\sqrt{\Phi_{i}^{\ast} d\mathrm{vol}_{g}}\right)_{(q,p)}e^{-\left\vert p\right\vert ^{2}/2t\hbar}d^{n}p,
\end{equation}
where $B(r)$ is the ball of radius $r$ in $\mathbb{R}^{n}$, then by (\ref{eqn:Ldecomp}) we have
\[
J(t)=\int_{M}\overline{\psi^{\prime}(q)}\,j_{t}(q)\,d\mathrm{vol}_{g}.
\]

The main result of this paper is the following theorem.

\begin{theorem}
\label{thm:main}
As $t\rightarrow0$,
\[
j_{t}^{r}(q)\sim\psi(q)-t\frac{\hbar}{2}\left(\Delta-\frac{1}{6}S\right)\psi(q)+O\left(t^{2}\right)  .
\]
\end{theorem}

\bigskip
Hence by (\ref{eqn:QE2}) and (\ref{eqn:Jrescaled}) we obtain%
\[
\tilde{Q}(E)=-\frac{\hbar^{2}}{2}(\Delta-\frac{1}{6}S).
\]

Before proving Theorem \ref{thm:main}, we show that the contribution to $J(t)$ coming from any region bounded away from the zero section is exponentially small.

\begin{theorem}
\label{thm:tails}Let $r_{0}\in(0,r)$, let $f$ be a function on $T^{\ast,r}M\setminus T^{\ast,r_{0}}M$, and suppose that for some $t$ and each $q\in M$,
\[
\left\vert \int_{T_{q}^{\ast,r}M\setminus T_{q}^{\ast,r_{0}}M}fe^{-E/t\hbar}d^{n}p\right\vert <\infty.
\]
Then as $t\rightarrow0$,
\[
\int_{T^{\ast,r}M\setminus T^{\ast,r_{0}}M}fe^{-E/t\hbar}\omega^{n}=O(e^{-r_{0}^{2}/2t\hbar}).
\]
\end{theorem}

\begin{proof}
Suppose first that $f\geq0.$ Then H\"{o}lder's inequality gives
\[
\int_{T_{q}^{\ast,r}M\setminus T_{q}^{\ast,r_{0}}M}fe^{-E/t\hbar}d^{n} p\leq\int_{T_{q}^{\ast,r}M\setminus T_{q}^{\ast,r_{0}}M}fe^{-E}d^{n}p\cdot\left\Vert e^{-(1/t\hbar-1)E}\right\Vert _{L^{\infty}(A_{q},d^{n}p)}.
\]
Since $\inf_{p:g^{jk}(q)p_{j}p_{k}\in(r_{0},r)}\{g^{jk}(q)p_{j}p_{k}/2\}=r_{0}^{2}/2$, the right-hand side above is equal to a possibly
$q$-dependent constant times $e^{-r_{0}^{2}/2t\hbar}.$

For general $f$, write $f=f_{+}-f\_$ where $f_{+}:=\max\{f,0\}$ and $f_{-}=-\min\{f,0\}.$ Then applying the first paragraph, we have
\[
\left\vert \int_{T_{q}^{\ast,r}M\setminus T_{q}^{\ast,r_{0}}M}fe^{-E/t\hbar}d^{n}p\right\vert \leq(C_{+}(q)+C_{-}(q))e^{-r_{0}^{2}/2t\hbar}.
\]
Since $M$ is compact, integrating over $M$ yields the result.
\end{proof}

\bigskip
Observe that Theorem \ref{thm:tails} is the same whether we take $r=\infty$ or $r<\infty$; this is why it is sufficient for our purposes to assume that the adapted complex structure only exists on a finite-radius tubular neighborhood of the zero section.

\bigskip
The remainder of this section is dedicated to the proof of Theorem \ref{thm:main}. Theorem \ref{thm:tails} allows us to apply Laplace's approximation to $j_{t}^{r}(q)$. Since $\psi$ and the metric $g$ are real analytic, they admit convergent Taylor series expansions in some neighborhood of each point, and our basic strategy is to Taylor expand the integrand $\psi\circ\pi\circ\Phi_{i}~\left(  \sqrt{d\mathrm{vol}_{g}},\sqrt{\Phi_{i}^{\ast}d\mathrm{vol}_{g}}\right)  $ around $x=0$. In order to insure that the expansions are valid, as we explain after (\ref{eqn:expDvol}) below, we must first replace $r$ by
\begin{equation}
\label{eqn:rprime}
r^{\prime}:=\min\{r,\min_{M}\{\sqrt{3/\left\Vert R_{jk}(q)\right\Vert }\}\},
\end{equation}
where $R_{jk}$ is the Ricci curvature tensor. We will then obtain expressions involving integrals of the form $\int_{-r^{\prime}}^{r^{\prime}}y^{m}e^{-y^{2}/2t}dy$. One checks that this is asymptotically equal to $\int_{\mathbb{R}}y^{m}e^{-y^{2}/2t}dy$ for small $t$. Then, since the integral $\int_{\mathbb{R}}y^{m}e^{-y^{2}/2t}dy$ is equal to $0$ if $m$ is odd and to $t^{k}\sqrt{2\pi t}(2k-1)!!$ if $m=2k$,
\begin{equation}
\label{eqn:gaussInt}
t^{-n/2}\int_{\mathbb{R}^{n}}p_{i_{1}}\cdots p_{i_{l}}e^{-\left\vert p\right\vert^{2}/2t\hbar}d^{n}p
\end{equation}
is equal to $\sqrt{2\pi\hbar}^{n}$ when $l=0$, vanishes when $l=1$ or $l=2$ and $i_{1}\neq i_{2}$, and is $O(t^{2})$ when $l\geq3$. In particular, to compute $\zeta_{1}$, we only need to keep track of terms up to order $O(\left\vert p\right\vert ^{3})$.

Since $X_{E}$ applied to the vertically constant function $\psi\circ\pi$ is $g^{jk}(x)p_{j}\frac{\partial\psi}{\partial x^{k}},$ the term $\psi\circ\pi\circ\Phi_{i}$ can be written as a Taylor series, centered at $q$, in the fiber coordinate $p$ as
\begin{equation}
\label{eqn:psiTaylor}
\psi\circ\pi\circ\Phi_{i}(q,p)=\psi(q)+ip^{k}\frac{\partial\psi}{\partial x^{k}}(q)-\frac{1}{2}p^{j}p^{k}\frac{\partial^{2}\psi}{\partial x^{j}\partial x^{k}}(q)+O(\left\vert p\right\vert ^{3}),
\end{equation}
where we write $p^{j}:=g^{jk}(q)p_{k}=\delta^{jk}p_{k}$ \cite{Hall-K_acs}. (At an arbitrary point $(x,p)$ in our coordinate neighborhood, the Taylor expansion has another factor involving derivatives of the metric, but the first-order derivatives of the metric vanish at $q$ in normal coordinates.) Note that $r$ was chosen so that this expansion is valid for $p\in T^{\ast,r}M$.

\bigskip
To compute $\left(  \sqrt{d\mathrm{vol}_{g}},\sqrt{\Phi_{i}^{\ast}d\mathrm{vol}_{g}}\right)  _{\text{BKS}}$ we need to compute the $dp$ term of $\Phi_{t}^{\ast} d\mathrm{vol}_{g}$, which is just the restriction of $\Phi_{t}^{\ast}d\mathrm{vol}_{g}$ to the vertical tangent space. Since the time-$1$ geodesic flow restricted to the vertical tangent space $V_{(q,p)}(T^{\ast}M)\simeq T_{q}^{\ast}M$ is the exponential map, the vertical part of $\Phi_{t}^{\ast}d\mathrm{vol}_{g}$ is
\begin{equation}
\label{eqn:phiDvol}
\left.  \Phi_{t}^{\ast}d\mathrm{vol}_{g}\right\vert _{T_{q}^{\ast}M}=\left.\left(  N_{t}^{\ast}\circ\exp_{q}^{\ast}\right)  d\mathrm{vol}_{g}\right\vert_{T^{\ast}M}.
\end{equation}
Leaving aside the $t$-rescaling for a moment, the Taylor expansion of the Riemannian volume form in normal coordinates is
\begin{equation}
\label{eqn:expDvol}
\sqrt{\det(g_{jk}(\exp_{q}(p)))}=1-\frac{1}{6}R_{jk}(q)p^{j}p^{k}+O\left(\left\vert p\right\vert ^{3}\right).
\end{equation}
This follows from the standard Taylor expansions $\det(g_{jk}(x))=1-\frac{1}{3}R_{jk}(q)x^{j}x^{k}+O(\left\vert x\right\vert ^{3})$ and
\begin{equation}
\label{eqn:TaylorSR}
\sqrt{1-x}=1-\frac{1}{2}x+O(x^{2})
\end{equation}
about $x=0$; since the latter is valid for $\left\vert x\right\vert <1$, (\ref{eqn:expDvol}) is valid for $\left\vert R_{jk}(q)p^{j}p^{k}\right\vert <3$, which explains (\ref{eqn:rprime}).

By Gauss' lemma \cite[Lemma 5.3]{doCarmo}, $\det((\exp_{q})_{\ast})_{p}=1$, whence
\[
d^{n}(\exp_{q}(p))=d^{n}p+dq\text{-terms.}
\]
Combining this ith (\ref{eqn:expDvol}), with $p$ replaced by $tp$ because of the rescaling in (\ref{eqn:phiDvol}), the analytic continuation of $\left. (\pi\circ\Phi_{t})^{\ast}d\mathrm{vol}_{g}\right\vert _{T_{q}^{\ast}M}$ to $t=\sqrt{-1}$ is
\[
\left(  d\mathrm{vol}_{g},(\pi\circ\Phi_{i})^{\ast}d\mathrm{vol}_{g}\right) =(2\pi\hbar)^{n}\left(  1+\frac{1}{6}p^{j}p^{k}R_{jk}(q)\right)  +O\left(\left\vert p\right\vert^{3}\right).
\]
After using (\ref{eqn:TaylorSR}) again, we obtain
\begin{equation}
\label{eqn:BKSpairing}
\sqrt{\left(d\mathrm{vol}_{g},(\pi\circ\Phi_{i})^{\ast}d\mathrm{vol}_{g}\right)} =(2\pi\hbar)^{n/2}\left(1+\frac{1}{12}p^{j}p^{k}R_{jk}(q)\right)+O\left(\left\vert p\right\vert^{3}\right).
\end{equation}

For any multi-index $\alpha$, one has $\int_{B(r^{\prime})}p^{\alpha}e^{-\left\vert p\right\vert ^{2}/2t\hbar}d^{n}p\sim\int_{\mathbb{R}^{n}}p^{\alpha}e^{-\left\vert p\right\vert ^{2}/2t\hbar}d^{n}p$ as $t\rightarrow 0$. Now put (\ref{eqn:psiTaylor}) and (\ref{eqn:BKSpairing}) into the integral $j_{t}^{r^{\prime}}(q)$ defined in (\ref{eqn:jtq}) and replace the integrals over $B(r^{\prime})$ with integrals over $\mathbb{R}^{n}$. Evaluating the resulting Gaussian integrals according to (\ref{eqn:gaussInt}) then yields
\begin{align*}
j_{t}^{r^{\prime}}(q) &  =\left(  2\pi\hbar t\right)  ^{-n/2}\int_{\mathbb{R}^{n}}\left(  \psi+ip^{k}\frac{\partial\psi}{\partial x^{k}}-\frac{1}{2}p^{j}p^{k}\frac{\partial^{2}\psi}{\partial x^{j}\partial x^{k}}\right)  \\
&  \qquad\qquad\qquad\qquad\qquad\times\left(  1+\frac{1}{12}p^{j}p^{k}R_{jk}\right)  e^{-\left\vert p\right\vert ^{2}/2t\hbar}d^{n}p+O(t^{2})\\
&  =\left(  2\pi\hbar t\right)  ^{-n/2}\left[  (2\pi\hbar t)^{n/2}\psi-(2\pi\hbar t)^{n/2}\frac{\delta^{jk}t\hbar}{2}\left(  \frac{\partial^{2}\psi
}{\partial x^{j}\partial x^{k}}-\frac{1}{6}R_{jk}\psi\right)  +O\left(t^{2}\right)  \right]  \\
&  =\left(  \psi-\frac{t\hbar}{2}\left(  \Delta-\frac{1}{6}S\right)\psi+O\left(  t^{2}\right)  \right)
\end{align*}
where $S(q)=R_{~j}^{j}(q)$ is the scalar curvature and $\Delta$ is the Laplacian (which, when evaluated at $q$ in normal coordinates centered at $q$,
is simply $\Delta\psi=\sum_{j}\partial^{2}\psi/\partial\left(  x^{j}\right)^{2}$).

The Wick-rotated ``dragging projection'' $\tilde{Q}(E)$ of the geodesic flow onto the vertically polarized quantum Hilbert space, defined in (\ref{eqn:QE2}), is therefore determined by requiring that for each $\psi^{\prime}\in L^{2}(M,d\mathrm{vol}_{g})$,
\begin{align*}
&\left\langle \psi^{\prime}, Q(E)\psi\right\rangle_{L^{2}(M,d\mathrm{vol}_{g})} \\
&  :=\hbar\left.  \frac{d}{dt}\right\vert _{t=0}\int_{M}\bar{\psi}^{\prime}(q)\left(  \psi(q)-\frac{t\hbar}{2}\left(  \Delta-\frac{1}{6}S(q)\right)  \psi(q)+O\left(  t^{2}\right)  \right)  d\mathrm{vol}_{g}\\
&  =\int_{M}\bar{\psi}^{\prime}\left[  -\frac{\hbar^{2}}{2}\left(\Delta-\frac{1}{6}S\right)  \psi\right]  d\mathrm{vol}_{g}
\end{align*}
which proves the theorem.

\bigskip
{\small
\providecommand{\bysame}{\leavevmode\hbox to3em{\hrulefill}\thinspace}
\providecommand{\MR}{\relax\ifhmode\unskip\space\fi MR }
\providecommand{\MRhref}[2]{%
  \href{http://www.ams.org/mathscinet-getitem?mr=#1}{#2}
}
\providecommand{\href}[2]{#2}

}

\begin{thebibliography}{DeW57}

\bibitem[AD99]{Driver-Andersson}
L. Andersson and B. Driver, \emph{Finite-dimensional approximations to Wiener measure and path integral formulas on manifolds},
  J. Funct. Anal. \textbf{165} (1999), no. 2, 430–-498.

\bibitem[BP08]{Baer-Pfaeffle}
C. B\"ar and F. Pf\"affle, \emph{Path Integrals on Manifolds by Finite Dimensional Approximation},
  J. Reine Angew. Math. \textbf{625} (2008), 29–-57.

\bibitem[BVN06]{Bastianelli-VanNieuwenhuizen}
F. Bastianelli and P. Van~Nieuwenhuizen, \emph{Path integrals and anomalies in curved space},
  Cambridge University Press, 2006.

\bibitem[Cha99]{CharlesPI}
L. Charles, \emph{Feynman path integral and {T}oeplitz quantization},
  Helv. Phys. Acta \textbf{72} (1999), no.~5-6, 341--355.

\bibitem[dC92]{doCarmo}
M.~Perdig{\~a}o do~Carmo, \emph{Riemannian geometry}, Mathematics: Theory
  \& Applications, Birkh\"auser Boston, Inc., Boston, MA, 1992, Translated from
  the second Portuguese edition by Francis Flaherty.

\bibitem[DeW57]{DeWitt}
B.~S. DeWitt, \emph{Dynamical theory in curved spaces. I. a review of the
  classical and quantum action principles}, Rev. Mod. Phys. \textbf{29} (1957),
  no.~3, 377 -- 397.

\bibitem[DK10]{Douglas-Klevtsov}
M.~R. Douglas and S. Klevtsov, \emph{Bergman kernel from path
  integral}, Comm. Math. Phys. \textbf{293} (2010), no.~1, 205--230.

\bibitem[Ful96]{Fulling}
S.~A. Fulling, \emph{Pseudodifferential operators, covariant quantization, the
  inescapable {V}an {V}leck-{M}orette determinant, and the {$R/6$}
  controversy}, Internat. J. Modern Phys. D \textbf{5} (1996), no.~6, 597--608,
  The Sixth Moscow Quantum Gravity Seminar (1995).

\bibitem[GS91]{Guillemin-Stenzel1}
V. Guillemin and M. Stenzel, \emph{Grauert tubes and the homogeneous
  {M}onge-{A}mp\`ere equation}, J. Differential Geom. \textbf{34} (1991),
  no.~2, 561--570.

\bibitem[GS92]{Guillemin-Stenzel2}
\bysame, \emph{Grauert tubes and the homogeneous {M}onge-{A}mp\`ere equation.
  {II}}, J. Differential Geom. \textbf{35} (1992), no.~3, 627--641.

\bibitem[HK11]{Hall-K_acs}
B.~C. Hall and W.~D. Kirwin, \emph{Adapted complex structures and the geodesic
  flow}, Math. Ann. \textbf{350} (2011), no.~2, 455 -- 474.

\bibitem[LS91]{Lempert-Szoke}
L.~Lempert and R.~Sz{\H{o}}ke, \emph{Global solutions of the homogeneous
  complex {M}onge-{A}mp\`ere equation and complex structures on the tangent
  bundle of {R}iemannian manifolds}, Math. Ann. \textbf{290} (1991), no.~4,
  689--712.

\bibitem[{\'S}ni80]{SniatyckiGQ}
J.~{\'S}niatycki, \emph{Geometric quantization and quantum mechanics}, Applied
  Mathematical Sciences, vol.~30, Springer-Verlag, New York-Berlin, 1980.

\bibitem[Sz{\H{o}}91]{Szoke}
R. Sz{\H{o}}ke, \emph{Complex structures on tangent bundles of
  {R}iemannian manifolds}, Math. Ann. \textbf{291} (1991), no.~3, 409--428.

\bibitem[P-W14]{Prat-Waldron}
A.~Prat-Waldron, \emph{private correspondence}.

\bibitem[Woo91]{Woodhouse}
N.~M.~J. Woodhouse, \emph{{Geometric Quantization, 2nd Edition}}, Oxford
  University Press, Inc., New York, 1991.

\end{thebibliography}
\end{document}